\documentclass{article}

\usepackage{amsmath, amssymb, amsfonts, amsbsy, amscd, latexsym, amsthm}

\date{}

\newtheorem{Theorem}{Theorem}[section]
\newtheorem{Proposition}[Theorem]{Proposition}

\newtheorem{Lemma}[Theorem]{Lemma}

\numberwithin{equation}{section}

\title{An invariance group for a linear combination
of two Saalsch\"utzian
${}_4F_3(1)$
hypergeometric series}

\author{Ilia D. Mishev
\footnote{Department of
Mathematics, University of Colorado at Boulder,
Campus Box 395, Boulder, CO 80309-0395,
U.S.A.
E-mail address: ilia.mishev@colorado.edu}}

\begin{document}

\maketitle

\begin{abstract}

We explore a function
$L(\vec{x})=L(a,b,c,d;e;f,g)$ which is a
linear combination of two Saalsch\"utzian
${}_4F_3(1)$
hypergeometric series. 
We demonstrate a fundamental two-term
relation satisfied by the $L$ function
and show that 
the fundamental two-term relation 
implies that the Coxeter 
group $W(D_5)$, 
which has 1920 elements, is an
invariance group for $L(\vec{x})$.
The invariance relations
for $L(\vec{x})$ 
are 
classified into six types based on a
double coset decomposition of the
invariance group. 
The fundamental two-term relation
is shown to generalize classical
results about hypergeometric series.
We derive
Thomae's identity for
${}_3F_2(1)$ series, Bailey's identity
for terminating Saalsch\"utzian
${}_4F_3(1)$ series, and Barnes' second
lemma as consequences
of the fundamental two-term relation.

\end{abstract}

\section{Introduction}

Invariance groups for hypergeometric series have
been studied extensively in the past. A hypergeometric 
series is trivially invariant under permutations
of its numerator and denominator parameters thus
giving us an invariance group isomorphic
to the cross product of two symmetric groups. 
The existence of nontrivial two-term relations 
and their combined use with
the trivial relations
leads to larger invariance
groups that have been the subject of study
over the last twenty-five years by
Beyer et al.\ \cite{BLS}, 
Srinivasa Rao et al.\ \cite{Rao}, and others.

The series of type ${}_3F_2(1)$ 
have been studied since
the nineteenth century. In 1879 Thomae
\cite{T} obtained a number of 
two-term relations for ${}_3F_2(1)$
series. One of those relations is known
today as Thomae's identity 
(see \cite[p.\ 14]{Ba}). 
Thomae's identity was later rediscovered 
(with an explicit proof provided)
by Ramanujan
(see \cite[p.\ 104]{Har}).
In 1923
Whipple \cite{Wh1} re-visited
Thomae's work and introduced 
a more
convenient notation, in terms of his 
Whipple parameters,
that indexed the
two-term relations
found by Thomae.
In a recent paper Krattenthaler and
Rivoal 
\cite{KR}
described other families of
two-term relations for ${}_3F_2(1)$
series that are not
consequences of the identities
found by Thomae.

A two-term relation for 
terminating Saalsch\"utzian
${}_4F_3(1)$ series, 
based on identities relating 
very-well-poised ${}_7F_6(1)$
series to terminating
Saalsch\"utzian ${}_4F_3(1)$ series,
was given
by Whipple \cite[Eq.\ (10.11)]{Wh3}
in 1925.
The same two-term relation 
appeared later in
Bailey's monograph
\cite[p.\ 56]{Ba} and is
often referred to today as Bailey's identity.

The first mention of an invariance
group for hypergeometric series seems
to be due to Hardy. In \cite[p.\ 111]{Har}
it is implied that 
the symmetric group $S_5$ is an invariance
group for the ${}_3F_2(1)$ series.
In 1987 Beyer et al.\ 
\cite{BLS} rediscovered that
Thomae's identity combined with the trivial
invariances under permutations of the
numerator and denominator parameters
implies that $S_5$ is an invariance group
for the ${}_3F_2(1)$ series. Beyer et al.\
also showed in the same paper 
\cite{BLS} that 
Bailey's identity combined with the 
trivial invariances
implies that the symmetric group $S_6$ is an
invariance group for the terminating
Saalsch\"utzian ${}_4F_3(1)$ series.

The goal of this paper is to extend the results 
stated 
above to Saalsch\"utzian ${}_4F_3(1)$ series.
We examine a function
$L(a,b,c,d;e;f,g)$ 
(see (\ref{220}) for the definition)
which is a linear combination
of two Saalsch\"utzian ${}_4F_3(1)$ series.
This particular linear combination
of two Saalsch\"utzian ${}_4F_3(1)$ series
appears in \cite{St2} in the evaluation
of the Mellin transform of a spherical
principal series $GL(4,\mathbb{R})$
Whittaker function.

In Section 3 we derive a 
fundamental two-term relation 
(see (\ref{340})) satisfied by
$L(a,b,c,d;e;f,g)$. 
The fundamental two-term relation (\ref{340})
is derived through a Barnes integral
representation of $L(a,b,c,d;e;f,g)$ and
generalizes both Thomae's and Bailey's
identities in the sense that the latter
two identities can be obtained
as limiting cases of our
fundamental two-term relation (see Section 5).  

In Section 4 we show that 
the two-term relation (\ref{340})
combined with the trivial invariances
of $L(a,b,c,d;e;f,g)$ under permutations
of $a,b,c,d$ and interchanging
$f,g$ implies that
the function $L(a,b,c,d;e;f,g)$ has an
invariance group $G_L$ isomorphic
to the Coxeter group $W(D_5)$,
which is of order 1920.
(See \cite{H} for general information
on Coxeter groups.)
The invariance group $G_L$ is 
given as a matrix
group of transformations of the affine
hyperplane
\begin{equation}
\label{110}
V=\{(a,b,c,d,e,f,g)^T \in
\mathbb{C}^7:
e+f+g-a-b-c-d=1\}.
\end{equation}
The 1920 invariances
of the $L$ function that follow 
from the invariance group $G_L$ are 
classified into six types based
on a double coset decomposition of $G_L$
with respect to its subgroup $\Sigma$ 
consisting of all the permutation matrices
in $G_L$. To the best of the author's
knowledge, using such a double coset
decomposition is a new way of describing 
all the relations 
induced by an invariance group
and does not have an analog in the literature
before.  

Some consequences of the fundamental
two-term relation (\ref{340}) are
shown in Section 5. In particular,
as already mentioned, we show that
Thomae's and Bailey's identities
follow as limiting cases of 
(\ref{340}). We also show that
Barnes' second lemma 
(see \cite{Bar2} or 
\cite[p.\ 42]{Ba})
follows as
a special case of (\ref{340})
when we take $d=g$.

Versions of the $L$ function 
(in terms of very-well-poised
${}_7F_6(1)$ series, see (\ref{230}))
were examined in the past
by Bailey \cite{Ba1},
Whipple \cite{Wh2}, 
and Raynal \cite{R}. 
Bailey obtained
two-term relations 
that were later re-visited by
Whipple and Raynal. However,
there is no mention
of an underlying invariance
group.
  
A basic hypergeometric series
analog of the $L$ function (in terms of
${}_8\phi_7$ series) was studied by 
Van der Jeugt and Srinivasa Rao \cite{V}.
The authors establish an invariance group 
isomorphic to $W(D_5)$
for the ${}_8\phi_7$ series, 
but do not classify all two-term
relations, or consider how they 
could imply results about
lower-order series.

Very recently
Formichella et al.\ \cite{FGS} explored
a function $K(a;b,c,d;e,f,g)$ which is
a different linear combination 
of two Saalsch\"utzian
${}_4F_3(1)$ series from 
the function $L(a,b,c,d;e;f,g)$. 
The linear combination
of two Saalsch\"utzian
${}_4F_3(1)$ series studied by
Formichella et al.\ appears in the
theory of archimedian zeta integrals
for automorphic $L$ functions 
(see \cite{St4,ST}).
The function
$K(a;b,c,d;e,f,g)$ behaves very differently
from $L(a,b,c,d;e;f,g)$. 
Formichella et al.\ obtain in \cite{FGS}
a two-term relation
satisfied by $K(a;b,c,d;e,f,g)$ and show that
their two-term relation implies
that the symmetric group $S_6$ is an
invariance group for $K(a;b,c,d;e,f,g)$.
In a future work by the author of the
present paper and by Green and Stade,
the connection between the $K$ and the $L$
functions will be studied. 

{\it Acknowledgments.} This paper is based on results
obtained in the author's Ph.D. thesis (see \cite{M})
at the University of Colorado at Boulder. The author
would like to acknowledge the guidance of his advisor 
Eric Stade as well as the discussions with R.M. Green from
the University of Colorado at Boulder and 
Robert S. Maier from
the University of Arizona.
  
\section{Hypergeometric series and Barnes integrals}

The hypergeometric series of type
${}_{p+1}F_p$ is the power series
in the complex variable $z$ defined by
\begin{equation}
\label{210}
{}_{p+1}F_p\left[
{\displaystyle a_1,a_2,\ldots, a_{p+1};
\atop
\displaystyle b_1,b_2,\ldots, b_p;}z\right]=
\sum_{n=0}^{\infty}\frac
{(a_1)_n(a_2)_n\cdots (a_{p+1})_n}
{n!(b_1)_n(b_2)_n\cdots (b_p)_n}z^n,
\end{equation}
where $p$ is a positive integer,
the numerator parameters 
$a_1,a_2,\ldots, a_{p+1}$ and the
denominator parameters
$b_1,b_2,\ldots, b_p$ are complex numbers,
and the rising factorial 
$(a)_n$ is given by
\begin{equation*}
(a)_n=\left\{
\begin{array}{rl}
a(a+1)\cdots(a+n-1)=
\frac{\Gamma(a+n)}{\Gamma(a)}, & n>0,\\
1, & n=0.
\end{array} \right.
\end{equation*}

The series in (\ref{210}) converges absolutely
if $|z|<1$. When $|z|=1$, the series converges
absolutely if $\mbox{Re}(\sum_{i=1}^{p} b_i - 
\sum_{i=1}^{p+1} a_i) > 0$ (see \cite[p.\ 8]{Ba}).
We assume that no denominator parameter is
a negative integer or zero. If a numerator 
parameter is a negative integer or zero, the
series has only finitely many nonzero terms and
is said to terminate. 

When $z=1$, the series is
said to be of unit argument and of type
${}_{p+1}F_p(1)$. If
$\sum_{i=1}^{p} b_i = 
\sum_{i=1}^{p+1} a_i + 1$, the series is called 
Saalsch\"utzian.
If $1+a_1=b_1+a_2= \ldots = b_p+a_{p+1}$,
the series is called well-poised. A well-poised
series that satisfies $a_2=1+\frac{1}{2}a_1$ is
called very-well-poised.

Our main object of study in this paper will be 
the function $L(a,b,c,d;e;f,g)$ defined by
\begin{eqnarray}
\label{220}
&&L(a,b,c,d;e;f,g) \nonumber \\
&&=\frac{{}_4F_3\left[
{\displaystyle a,b,c,d;
\atop
\displaystyle e,f,g;}1\right]}{\sin \pi e\ \Gamma(e)
\Gamma(f)\Gamma(g)\Gamma(1+a-e)\Gamma(1+b-e)\Gamma(1+c-e)
\Gamma(1+d-e)} \nonumber \\
&&-\frac{{}_4F_3\left[
{\displaystyle 1+a-e,1+b-e,1+c-e,1+d-e;
\atop
\displaystyle 1+f-e,1+g-e,2-e;}1\right]}
{\sin \pi e\ \Gamma(a)
\Gamma(b)\Gamma(c)\Gamma(d)\Gamma(1+f-e)\Gamma(1+g-e)
\Gamma(2-e)},
\end{eqnarray}
where $a,b,c,d,e,f,g \in \mathbb{C}$ satisfy
$e+f+g-a-b-c-d=1$.

The function $L(a,b,c,d;e;f,g)$ is a linear combination
of two Saalsch\"utzian ${}_4F_3(1)$ series.
Other notations we will use for 
$L(a,b,c,d;e;f,g)$ are
$L\left[{\displaystyle a,b,c,d; \atop
\displaystyle e;f,g}\right]$ and
$L(\vec{x})$, where 
we will always have
$\vec{x}=(a,b,c,d,e,f,g)^{T} \in V$
(see (\ref{110})).

It should be noted that by 
\cite[Eq.\ (7.5.3)]{Ba},
the L function can be expressed as
a very-well-poised 
${}_7F_6(1)$ series:
\begin{align}
\label{230}
&L(a,b,c,d;e;f,g) \nonumber\\
{}\\
&=\frac{\Gamma(1+d+g-e)}
{\pi \Gamma(g)\Gamma(1+g-e)
\Gamma(f-d)\Gamma(1+a+d-e)
\Gamma(1+b+d-e)\Gamma(1+c+d-e)}
\nonumber\\
&\cdot {}_7F_6\left[
{\displaystyle d+g-e,
1+\frac{1}{2}(d+g-e),g-a,g-b,g-c,
d,1+d-e;
\atop
\displaystyle \frac{1}{2}(d+g-e),
1+a+d-e,1+b+d-e,1+c+d-e,1+g-e,g;}
1\right],\nonumber
\end{align}
provided that $\mbox{Re}(f-d)>0$.
Therefore our results on the $L$
function can also be interpreted in terms
of the very-well-poised
${}_7F_6(1)$ series given in
(\ref{230}). 

Fundamental to the derivation
of a nontrivial two-term relation 
for the
$L$ function will be the notion
of a Barnes integral, which
is a contour integral of the
form
\begin{equation}
\label{240}
\int_t \prod_{i=1}^{n}\Gamma^{\epsilon_i}(a_i+t)
\prod_{j=1}^{m}\Gamma^{\epsilon_j}(b_j-t) \,dt,
\end{equation}
where $n,m \in \mathbb{Z}^{+};
\epsilon_i,\epsilon_j = \pm 1;$
and $a_i,b_j,t \in \mathbb{C}$.
The path of integration is
the imaginary axis,
indented if necessary, so that any poles
of $\prod_{i=1}^{n}\Gamma^{\epsilon_i}(a_i+t)$
are to the left of the contour and
any poles of
$\prod_{j=1}^{m}\Gamma^{\epsilon_j}(b_j-t)$ are to
the right of the contour. This path of integration
always exists, provided that, for
$1 \leq i \leq n$ and $1 \leq j \leq m$, we
have $a_i+b_j \notin \mathbb{Z}$ whenever
$\epsilon_i=\epsilon_j=1$. 

From now on, when we write an integral
of the form (\ref{240}),
we will always
mean a Barnes integral with a path of integration
as just described.

A Barnes integral can often be evaluated in terms of
hypergeometric series using the
Residue Theorem,
provided that we can establish the
necessary convergence arguments.
This is the approach we
take in the next section. 
We will make use of the extension of
Stirling's formula to the complex numbers 
(see \cite[Section 4.42]{Titch} or \cite[Section 13.6]{WW}):
\begin{equation}
\label{250} \Gamma (a+z) = \sqrt{2\pi} z^{a+z-1/2}e^{-z} (1+
\mbox{O} (1/|z|))
\mbox{ uniformly as } |z| \to \infty,
\end{equation}
provided that 
$-\pi +\delta \leq \arg (z) \leq \pi -
\delta, \; \delta \in (0,\pi)$.

When applying the Residue Theorem,
we will use 
the fact
that the gamma 
function has simple poles at
$t=-n, n=0,1,2,\dots$, with
\begin{equation}
\label{260}
\mbox{Res}_{t=-n}\Gamma(t)=\frac{(-1)^n}{n!}.
\end{equation}

When simplifying expressions
involving gamma functions, 
the reflection formula for the
gamma function will
often be
used:
\begin{equation}
\label{270}
\Gamma(t)\Gamma(1-t)=
\frac{\pi}{\sin \pi t}.
\end{equation}

Finally, we will use
a result about Barnes integrals
known as Barnes' lemma
(see \cite{Bar1} or \cite[p.\ 6]{Ba}):

\begin{Lemma}[Barnes' lemma]
\label{BL}

If
$\alpha,\beta,\gamma,\delta \in \mathbb{C}$,
we have
\begin{eqnarray}
\label{290}
\frac{1}{2 \pi i} \int_t
\Gamma(\alpha+t) \Gamma(\beta+t)
\Gamma(\gamma-t) \Gamma(\delta-t) \,dt
\nonumber\\
=\frac{\Gamma(\alpha+\gamma)\Gamma(\alpha+\delta)
\Gamma(\beta+\gamma)\Gamma(\beta+\delta)}
{\Gamma(\alpha+\beta+\gamma+\delta)},
\end{eqnarray}
provided that none of 
$\alpha+\gamma,\alpha+\delta,
\beta+\gamma$ and $\beta+\delta$ is an
integer.

\end{Lemma}

\section{Fundamental two-term relation}

In this section we show that the function
$L(a,b,c,d;e;f,g)$ defined in (\ref{220}) can
be represented as a Barnes integral. The Barnes integral
representation will then be used to derive a
fundamental two-term relation satisfied by the $L$
function.

\begin{Proposition}
\label{P310}
\begin{eqnarray}
\label{310}
&&L(a,b,c,d;e;f,g) \nonumber \\
&&=\frac{1}{\pi \Gamma(a)\Gamma(b)\Gamma(c)\Gamma(d)
\Gamma(1+a-e)\Gamma(1+b-e)\Gamma(1+c-e)\Gamma(1+d-e)}
\nonumber \\
&&\cdot \frac{1}{2\pi i}
\int_t\frac{\Gamma(a+t)\Gamma(b+t)\Gamma(c+t)\Gamma(d+t)
\Gamma(1-e-t)\Gamma(-t)}{\Gamma(f+t)\Gamma(g+t)}\,dt.
\end{eqnarray}

\end{Proposition}

In the proof of Proposition \ref{P310}, we will 
need the following statement.

\begin{Lemma}
\label{L320} 

For every $\varepsilon >0 $, there is a constant
$K=K(\varepsilon) $, such that if 
$\emph{dist} (z, \mathbb{Z}) \geq
\varepsilon $,  then
\begin{equation}
\label{320} |\sin \pi z |\geq K e^{\pi |\emph{Im} (z)|}.
\end{equation}

\end{Lemma}

\begin{proof}

Let $z=x+iy$. We have $$ \sin \pi z = \frac{1}{2i} \left (
e^{i \pi (x+iy)}-  e^{-i \pi (x+iy)} \right ) = \sin \pi x  \cosh
\pi y + i \cos \pi x \sinh \pi y. $$ 

Since $|\sinh \pi y| \leq
\cosh \pi y $, it follows that $ \sinh \pi|y|  \leq |\sin \pi z| \leq
\cosh \pi y. $

We may assume that $\varepsilon \in (0,1).$  If 
$\mbox{dist} (z, \mathbb{Z})
\geq \varepsilon,$ then at least one of the following 
two statements holds:

(a) $\mbox{dist} (x, \mathbb{Z}) \geq \varepsilon/2.$   

(b) $|y| \geq
\varepsilon/2.$  

If (a) holds, then
$$|\sin \pi z | \geq |\sin \pi x | \cosh
\pi y  \geq  \sin (\pi \varepsilon/2) \cosh \pi y \geq
\frac{1}{2}\sin (\pi \varepsilon/2)  e^{\pi |y|}.$$ 

If (b) holds,
then 
$$|\sin \pi z | \geq \sinh \pi |y| = \frac{1}{2}
e^{\pi|y|}(1-e^{-2\pi |y|}) \geq \frac{1}{2}(1-e^{-\pi \varepsilon})
e^{\pi|y|}. $$ 

Thus (\ref{320}) holds with $K= \frac{1}{2} \min
\{\sin (\pi \varepsilon/2),1-e^{-\pi \varepsilon}\}.$

\end{proof}

\begin{proof}[Proof of Proposition \ref{P310}]

Let
\begin{eqnarray}
\label{330}
&&I\left[
{\displaystyle a,b,c,d;
\atop
\displaystyle e;f,g}\right]\nonumber\\
&&=\frac{1}{2\pi i}
\int_t\frac{\Gamma(a+t)\Gamma(b+t)\Gamma(c+t)\Gamma(d+t)
\Gamma(1-e-t)\Gamma(-t)}{\Gamma(f+t)\Gamma(g+t)}\,dt.
\end{eqnarray}

For $N \geq 1$, let $C_N$ be the semicircle of radius $\rho_N  $ on
the right side of the imaginary axis and center at the origin, chosen
in such a way that $\rho_N \to \infty$ as $N \to \infty$ and
$$
\varepsilon:= \inf_N  
\mbox{dist}(C_N, \mathbb{Z}\cup (\mathbb{Z}-e)) >0.
$$

The formula
(\ref{270}) gives
\begin{eqnarray*}
G(t):=\frac{\Gamma(a+t)\Gamma(b+t)\Gamma(c+t)\Gamma(d+t)
\Gamma(1-e-t)\Gamma(-t)}{\Gamma(f+t)\Gamma(g+t)}\\
=\frac{-\pi^2\Gamma(a+t)\Gamma(b+t)
\Gamma(c+t)\Gamma(d+t)}{\Gamma(f+t)\Gamma(g+t)
\Gamma(e+t)\Gamma(1+t)\sin \pi t \sin \pi (e+t)}.
\end{eqnarray*}

By Stirling's formula (\ref{250}),
$$\frac{\Gamma(a+t)\Gamma(b+t)
\Gamma(c+t)\Gamma(d+t)}{\Gamma(f+t)\Gamma(g+t)
\Gamma(e+t)\Gamma(1+t)} \sim
t^{a+b+c+d-e-f-g-1}=t^{-2}.$$

By Lemma \ref{L320}, 
there exists a constant $K= K(\varepsilon)$ such
that
$$\frac{1}{|\sin \pi t  \sin \pi (e+t)|} \leq \frac{1}{K^2}\quad
\text{if} \;\; t \in C_N, \; N= 1,2, \ldots .$$

Therefore we obtain by the above estimates that there is a
constant $\tilde{K} >0 $ such that
$$ |G(t)| \leq \tilde{K}/|t|^2 \quad \text{if} \; \; t \in C_N, \;\; N=1,2,\ldots. $$

Thus $$ \left | \int_{C_N} G(t) \,dt \right | \leq
\frac{\tilde{K}}{\rho_N^2} \cdot \pi \rho_N    \to 0 \quad \mbox{as }
N \to \infty,$$ which implies
$$\int_{C_N} G(t) \,dt \to 0 \quad \mbox{as } N \to \infty.$$

It follows that the integral given by
$I\left[
{\displaystyle a,b,c,d;
\atop
\displaystyle e;f,g}\right]$ 
is equal to the sum of the
residues of the poles of
$\Gamma(1-e-t)$ and $\Gamma(-t)$. Adding up
the residues and making use of 
(\ref{270}), we obtain
\begin{eqnarray*}
I\left[
{\displaystyle a,b,c,d;
\atop
\displaystyle e;f,g}\right]
=\frac{\pi\Gamma(a)\Gamma(b)\Gamma(c)\Gamma(d)}
{\sin \pi e\ \Gamma(e)\Gamma(f)\Gamma(g)}
{}_4F_3\left[
{\displaystyle a,b,c,d;
\atop
\displaystyle e,f,g;}1\right]\\
-\frac{\pi \Gamma(1+a-e)\Gamma(1+b-e)\Gamma(1+c-e)
\Gamma(1+d-e)}
{\sin \pi e\ \Gamma(1+f-e)\Gamma(1+g-e)\Gamma(2-e)}\\
\cdot {}_4F_3\left[
{\displaystyle 1+a-e, 1+b-e, 1+c-e, 1+d-e;
\atop
\displaystyle 1+f-e,1+g-e,2-e;}1\right],
\end{eqnarray*}
from which the result follows.

\end{proof}

The fundamental two-term relation satisfied
by $L(a,b,c,d;e;f,g)$ is given in the
next proposition.

\begin{Proposition}
\label{P330}

\begin{equation}
\label{340}
L\left[
{\displaystyle a,b,c,d;
\atop
\displaystyle e;f,g}\right]
=L\left[
{\displaystyle a,b,g-c,g-d;
\atop 
\displaystyle 1+a+b-f;1+a+b-e,g}\right].
\end{equation}

\end{Proposition}

\begin{proof}

Let
$I\left[
{\displaystyle a,b,c,d;
\atop
\displaystyle e;f,g}\right]$
be as given in (\ref{330}). As a first step,
we will prove that 
\begin{eqnarray}
\label{350}
\frac{I\left[
{\displaystyle a,b,c,d;
\atop
\displaystyle e;f,g}\right]}{\Gamma(c)\Gamma(d)
\Gamma(1+a-e)\Gamma(1+b-e)}\nonumber\\
=\frac{I\left[
{\displaystyle a,b,g-c,g-d;
\atop
\displaystyle 1+a+b-f;1+a+b-e,g}\right]}{\Gamma(f-a)\Gamma(f-b)
\Gamma(g-c)\Gamma(g-d)}.
\end{eqnarray}

By Barnes' lemma,
\begin{align*}
&\frac{\Gamma(a+t)\Gamma(b+t)}
{\Gamma(f+t)}\\
=\frac{1}{2 \pi i\Gamma(f-a)\Gamma(f-b)}
\int_{u}&\Gamma(t+u)\Gamma(f-a-b+u)
\Gamma(a-u)\Gamma(b-u)\,du
\end{align*}
and
\begin{align*}
&\frac{\Gamma(c+t)\Gamma(d+t)}{\Gamma(g+t)}\\
=\frac{1}{2 \pi i\Gamma(g-c)\Gamma(g-d)}
\int_{v}&\Gamma(t+v)\Gamma(g-c-d+v)
\Gamma(c-v)\Gamma(d-v)\,dv.
\end{align*}

We re-write the integral for
$I\left[
{\displaystyle a,b,c,d;
\atop
\displaystyle e;f,g}\right]$ by substituting for the above expressions, 
changing the order of integration, so that we integrate 
with respect to $t$ first, and then applying Barnes' Lemma again 
to the integral with respect to $t$. We obtain
\begin{eqnarray}
\label{360}
&&\frac{I\left[
{\displaystyle a,b,c,d;
\atop
\displaystyle e;f,g}\right]}{\Gamma(c)\Gamma(d)
\Gamma(1+a-e)\Gamma(1+b-e)}\nonumber\\
&&=\frac{-1}
{4 \pi^2 \Gamma(c)\Gamma(d)
\Gamma(1+a-e)\Gamma(1+b-e)
\Gamma(f-a)\Gamma(f-b)
\Gamma(g-c)\Gamma(g-d)}\nonumber\\
&&\cdot
\int_{u} \Gamma(f-a-b+u)\Gamma(a-u)
\Gamma(b-u)\Gamma(u)\Gamma(1-e+u)\nonumber\\
&&\cdot \left( \int_{v} \frac
{\Gamma(g-c-d+v)\Gamma(c-v)
\Gamma(d-v)\Gamma(v)\Gamma(1-e+v)}
{\Gamma(1-e+u+v)}\,dv \right) du.
\end{eqnarray}

After the substitution 
$v \mapsto c+d-f+v$ in the inside integral,
it is easily checked 
(using the Saalsch\"utzian
condition
$e+f+g-a-b-c-d=1$)
that the
right-hand side of (\ref{360})
is invariant under the transformation
$$(a,b,c,d;e;f,g) \mapsto
(a,b,g-c,g-d;1+a+b-f;1+a+b-e,g),$$
which proves (\ref{350}). The result
in the proposition now follows 
immediately from 
(\ref{350}) upon writing the two $L$ functions
in (\ref{340}) in terms of their
Barnes integral representations (\ref{310}).

\end{proof}

\section{Invariance group}

In the previous section we showed that
the function $L(a,b,c,d;e;f,g)$ satisfies 
the two-term relation (\ref{340}). If we
define
\begin{equation}
\label{410}
A=\begin{pmatrix}
1 & 0 & 0 & 0 & 0 & 0 & 0\\
0 & 1 & 0 & 0 & 0 & 0 & 0\\
0 & 0 & -1 & 0 & 0 & 0 & 1\\
0 & 0 & 0 & -1 & 0 & 0 & 1\\
0 & 0 & -1 & -1 & 1 & 0 & 1\\
0 & 0 & -1 & -1 & 0 & 1 & 1\\
0 & 0 & 0 & 0 & 0 & 0 & 1\\
\end{pmatrix} 
\in GL(7,\mathbb{C}),
\end{equation}
then (\ref{340}) can be
expressed as
$L(\vec{x})=L(A\vec{x})$.

If $\sigma \in S_7$, we will
identify $\sigma$ with the matrix
in $GL(7,\mathbb{C})$ that
permutes the standard basis
$\{e_1,e_2,\ldots,e_7\}$ of the 
complex vector
space $\mathbb{C}^7$
according to the permutation
$\sigma$. For example,
$$(123)= \begin{pmatrix}
0 & 0 & 1 & 0 & 0 & 0 & 0\\
1 & 0 & 0 & 0 & 0 & 0 & 0\\
0 & 1 & 0 & 0 & 0 & 0 & 0\\
0 & 0 & 0 & 1 & 0 & 0 & 0\\
0 & 0 & 0 & 0 & 1 & 0 & 0\\
0 & 0 & 0 & 0 & 0 & 1 & 0\\
0 & 0 & 0 & 0 & 0 & 0 & 1\\
\end{pmatrix}.$$

Let
\begin{equation}
\label{420}
G_L=\langle (12),(23),
(34),(67),A \rangle
\leq GL(7,\mathbb{C}).
\end{equation}
The two-term relation
(\ref{340}) along 
with the trivial
invariances of 
the function
$L(a,b,c,d;e;f,g)$
under permutations of $a,b,c,d$ and
interchanging $f,g$ implies that
$G_L$ is an invariance group
for $L(a,b,c,d;e;f,g)$, i.e.
$L(\vec{x})=L(\alpha \vec{x})$
for every $\alpha \in G_L$.

The goal of this section is to
find the isomorphism type of
the group $G_L$ and further to 
describe the two-term relations
for the $L$ function in terms of
a double coset decomposition
of $G_L$ with respect to its
subgroup $\Sigma$ defined as
follows:
\begin{equation}
\label{430}
\Sigma=\langle (12),(23),
(34),(67) \rangle.
\end{equation}
The group $\Sigma$ is a
subgroup of $G_L$ consisting
of permutation matrices. It is
clear that $\Sigma \cong S_4 \times S_2$
and so $|\Sigma|=48$.
We note that if $\sigma \in \Sigma,
\alpha \in G_L$, the multiplication
$\sigma\alpha$ permutes the rows of
$\alpha$, and the multiplication
$\alpha\sigma$ permutes the columns of
$\alpha$.
A double coset of 
$\Sigma$ in $G_L$ 
is a set of the form
\begin{equation}
\label{440}
\Sigma \alpha \Sigma=
\{\sigma \alpha \tau : \sigma,\tau \in \Sigma\},
\mbox{ for some } \alpha \in G_L.
\end{equation}
The distinct double cosets
of the form (\ref{440}) partition
the group $G_L$ and give us
a double coset decomposition of
$G_L$ with respect to $\Sigma$. (See 
\cite[p. 119]{DF} for
more on double cosets.)

In Theorem \ref{T410} below we show
that the group $G_L$ is isomorphic
to the Coxeter group $W(D_5)$,
which is of
order 1920. In Theorem \ref{420}
we show that the subgroup $\Sigma$
is the largest permutation subgroup
of $G_L$ and obtain a double coset
decomposition of $G_L$ with respect to
$\Sigma$.  
We list a representative for
each of the six double cosets obtained
and give the six invariance relations
induced by those representatives
(see (\ref{450})--(\ref{460})). 
The six invariance relations 
(\ref{450})--(\ref{460})
listed are
all the ``different" types of invariance
relations in the sense that every other invariance
relation can be obtained by permuting
the first four entries and 
permuting
the last two entries 
on the right-hand side
of a listed invariance relation 
(which corresponds
to permuting the rows of the accompanying matrix), and
by permuting $a,b,c,d$ and  permuting $f,g$
on the right-hand side
of a listed invariance relation 
(which corresponds to permuting the
columns of the accompanying matrix).
  
\begin{Theorem}
\label{T410}

The group $G_L$ is isomorphic to the Coxeter group 
$W(D_5)$, which is of order 1920.

\end{Theorem}

\begin{proof}

The Dynkin diagram of the Coxeter
group $W(D_5)$ is given by the graph
with vertices labeled 
$1,2,3,4,1'$, where
$i,j \in \{1,2,3,4\}$ are connected
by an edge if 
and only if $|i-j|=1$, and
$1'$ is connected to $2$ only.
The presentation of $W(D_5)$ is
given by
$$W(D_5)=\langle s_1,s_2,s_3,s_4,s_{1'}:
(s_is_j)^{m_{ij}}=1 \rangle,$$
where
$m_{ii}=1$ for all $i$; and 
for $i$ and $j$ distinct,
$m_{ij}=3$ if $i$ and $j$ are
connected by an edge, and 
$m_{ij}$=2 otherwise.
It is well-known that the order of $W(D_5)$ is
$2^4\cdot5!=1920$ (see \cite[Section $2.11$]{H}). 

Consider the elements of $G_L$ given by
\begin{equation}
\label{470}
a_1=(34),a_2=(23),a_3=(34)A,
a_4=(67), a_{1'}=(12).
\end{equation}
It is clear $G_L=\langle a_i:
i \in \{1,2,3,4,1'\} \rangle$.
A direct computation shows that 
$$(a_ia_j)^{m_{ij}}=1, \quad \mbox{for all }
i,j \in \{1,2,3,4,1'\}.$$

Therefore if we define $\varphi(s_i)=a_i$ for 
every $i \in \{1,2,3,4,1'\}$, $\varphi$
extends (uniquely) to a surjective homomorphism
from $W(D_5)$ onto $G_L$
(see \cite[Section $1.6$]{DF}). 
Since $W(D_5)$ is a finite
group, if we show that $G_L$ and $W(D_5)$ have
the same order, it will follow that $\varphi$ is
an isomorphism and 
the theorem will be proved.

Since $\varphi$ is a surjective
homomorphism, the First Isomorphism Theorem for 
groups 
(see \cite[p.\ 98]{DF})
implies that
$|G_L| = |\mbox{Im}(\varphi)|$ must divide
$|W(D_5)|=1920$. Therefore if we show that
$|G_L| > 960=\frac{1920}{2}$,
then necessarily $|G_L|=1920$. We will
obtain an estimate on the order of $G_L$
by computing the sizes of the double cosets
$\Sigma A \Sigma$ and 
$\Sigma ((123)(67)A)^2 \Sigma$
of $\Sigma$ in $G_L$,
where $\Sigma$ is as given in (\ref{430}).

The matrix $A$ is given by
$$A=\begin{pmatrix}
1 & 0 & 0 & 0 & 0 & 0 & 0\\
0 & 1 & 0 & 0 & 0 & 0 & 0\\
0 & 0 & -1 & 0 & 0 & 0 & 1\\
0 & 0 & 0 & -1 & 0 & 0 & 1\\
0 & 0 & -1 & -1 & 1 & 0 & 1\\
0 & 0 & -1 & -1 & 0 & 1 & 1\\
0 & 0 & 0 & 0 & 0 & 0 & 1\\
\end{pmatrix}.$$
We see that all the rows of $A$ are distinct
as sequences.
Therefore multiplying 
$A$ on the left by $\sigma$, for
$\sigma \in \Sigma$, will give us 48 
matrices in $G_L$ that
belong to the double coset $\Sigma A \Sigma$. 
We note that the products $\sigma A$, for
$\sigma \in \Sigma$, amount to obtaining
all possible permutations of the first four rows
of $A$ and all possible permutations of the
last two rows of $A$.
By considering products
of the form $A\sigma$, for
$\sigma \in \Sigma$, we can permute the first
four columns of $A$ and the last two columns of $A$ 
in every possible way. If we first
permute columns of $A$ 
that
are different as multisets, and then permute the
rows of the resulting matrix in all 48 different ways,
we obtain 48 new elements in $G_L$
that belong to the double coset $\Sigma A \Sigma$.
Now,
the first and second columns of $A$ are the same 
as multisets and so are the third and the fourth columns. 
Thus we
permute the first four columns in 
$\frac{4!}{2!2!}=6$ different ways. The
sixth and seventh columns of $A$ are 
different as 
multisets and so we permute
them in 2 different ways.
In total, we permute the columns of $A$ in
$6 \cdot 2=12$ different ways and then we permute the rows
of each of the resulting matrices in all 48 possible ways to
obtain 
that the number of matrices that
belong to the 
double coset $\Sigma A \Sigma$
is $12 \cdot 48$.

Next we consider the matrix
$$A_1=((123)(67)A)^2=
\begin{pmatrix}
0 & -1 & -1 & -1 & 0 & 1 & 1\\
0 & 0 & -1 & 0 & 0 & 0 & 1\\
1 & 0 & 0 & 0 & 0 & 0 & 0\\
0 & 0 & -1 & 0 & 0 & 1 & 0\\
0 & -1 & -2 & -1 & 1 & 1 & 1\\
0 & 0 & -1 & -1 & 0 & 1 & 1\\
0 & -1 & -1 & 0 & 0 & 1 & 1\\
\end{pmatrix}.$$
We see that
$A_1$ contains an entry of $-2$, which is not the case
with $A$, implying that 
the double cosets $\Sigma A_1 \Sigma$ and
$\Sigma A \Sigma$ are distinct.
All the rows of $A_1$ are distinct
as sequences.
The first, second and third columns of $A_1$ are 
different as multisets and the fourth column represents
the same multiset as the second column. 
The sixth and seventh 
columns of $A_1$ are the same as multisets. 
Thus we permute the columns of $A_1$ in
$\frac{4!}{2!}=12$ different ways 
and then we permute the rows
of each of the resulting matrices in all 
48 possible ways to
obtain
that the number of matrices that
belong to the 
double coset $\Sigma A_1 \Sigma$
is $12 \cdot 48$.

Considering 
the number of matrices that
belong to the 
double cosets $\Sigma A \Sigma$
and $\Sigma A_1 \Sigma$, 
we see that the group $G_L$ contains
at least $12 \cdot 48 + 12 \cdot 48
> 960$ elements. 
Therefore $|G_L|=|W(D_5)|$ and the theorem
is proved.

\end{proof}

As stated before Theorem \ref{T410},
we are interested in the complete
double coset decomposition of $G_L$
with respect to $\Sigma$ since this
will classify all the invariance
relations for the function
$L(a,b,c,d;e;f,g)$ in a convenient way.
We use the same  
technique as in the proof of
Theorem \ref{T410} 
given by permuting columns that are different as
multisets and then permuting the rows of the resulting 
matrices
in every possible way. 
We obtain that there are
six double cosets of $\Sigma$ in $G_L$.
Representative matrices for the double cosets are
$I_7, A, ((123)(67)A)^2, 
((123)(67)A)^3, ((123)A)^3,
((123)(67)A)^4$. The corresponding
double coset sizes are 
$1 \cdot 48, 12 \cdot 48, 12 \cdot 48,
12 \cdot 48, 2 \cdot 48, 1 \cdot 48$.
Furthermore, the representative
matrices are all seen to have different entries (as,
for example, 
we determined for the matrices $A$ and $((123)(67)A)^2$
in the proof of Theorem \ref{410}) so that
$\Sigma$ must indeed be the largest permutation
subgroup of $G_L$. 
Each representative matrix gives rise to
an invariance relation. 
Theorem \ref{T420} summarizes the result.
 
\begin{Theorem}
\label{T420}

Let 
$\Sigma$ be as defined in (\ref{430}). Then
$\Sigma$ consists of all the permutation matrices
in $G_L$. There are six double
cosets in the double coset decomposition
of $G_L$ with respect to $\Sigma$.
Representative matrices for the double cosets
are $I_7, A, ((123)(67)A)^2, 
((123)(67)A)^3, ((123)A)^3,
((123)(67)A)^4$ and the corresponding
double coset sizes are
$1 \cdot 48, 12 \cdot 48, 12 \cdot 48,
12 \cdot 48, 2 \cdot 48, 1 \cdot 48$.
The corresponding invariances of the $L$ function are
given by
\begin{align}
L\left[
{\displaystyle a,b,c,d;
\atop
\displaystyle e;f,g}\right] &=
L\left[
{\displaystyle a,b,c,d;
\atop
\displaystyle e;f,g}\right],
\label{450}\\
L\left[
{\displaystyle a,b,c,d;
\atop
\displaystyle e;f,g}\right] &=
L\left[
{\displaystyle a,b,g-c,g-d;
\atop
\displaystyle 1+a+b-f;1+a+b-e,g}\right],
\label{452}\\
L\left[
{\displaystyle a,b,c,d;
\atop
\displaystyle e;f,g}\right] &=
L\left[
{\displaystyle 1+a-e,g-c,a,f-c;
\atop
\displaystyle 1+a-c;1+a+b-e,1+a+d-e}\right],
\label{454}\\
L\left[
{\displaystyle a,b,c,d;
\atop
\displaystyle e;f,g}\right] &=
L\left[
{\displaystyle 1+d-e,1+a-e,g-c,g-b;
\atop
\displaystyle 1+g-b-c;1+a+d-e,1+g-e}\right],
\label{456}\\
L\left[
{\displaystyle a,b,c,d;
\atop
\displaystyle e;f,g}\right] &=
L\left[
{\displaystyle g-a,g-b,g-c,g-d;
\atop
\displaystyle 1+g-f;1+g-e,g}\right],
\label{458}\\
L\left[
{\displaystyle a,b,c,d;
\atop
\displaystyle e;f,g}\right] &=
L\left[
{\displaystyle 1+c-e,1+d-e,1+a-e,1+b-e;
\atop
\displaystyle 2-e;1+g-e,1+f-e}\right].
\label{460}
\end{align}

\end{Theorem}

\section{Applications of the
fundamental two-term relation}

In this final section we prove
some consequences of the fundamental
two-term relation given in
Proposition \ref{P330}. As a first
step, we write 
the two $L$ functions
in (\ref{340}) in terms of
their definitions 
as linear combinations of two
${}_4F_3(1)$ series. 
We obtain
\begin{eqnarray}
\label{510}
\frac{{}_4F_3\left[
{\displaystyle a,b,c,d;
\atop
\displaystyle e,f,g;}1\right]}
{\sin \pi e\ \Gamma(e)\Gamma(f)
\Gamma(g)\Gamma(1+a-e)
\Gamma(1+b-e)\Gamma(1+c-e)
\Gamma(1+d-e)}\nonumber\\
-\frac{{}_4F_3\left[
{\displaystyle 1+a-e,1+b-e,
1+c-e,1+d-e;
\atop
\displaystyle 1+f-e,
1+g-e,2-e;}1\right]}
{\sin \pi e\ \Gamma(a)\Gamma(b)
\Gamma(c)\Gamma(d)\Gamma(1+f-e)
\Gamma(1+g-e)\Gamma(2-e)}\nonumber\\
=\frac{{}_4F_3\left[
{\displaystyle a,b,g-c,g-d;
\atop
\displaystyle 1+a+b-f,
1+a+b-e,g;}1\right]}
{\left[
{\displaystyle \sin \pi (1+a+b-f) 
\Gamma(1+a+b-f)
\Gamma(1+a+b-e)\Gamma(g)
\atop
\displaystyle \cdot
\Gamma(f-b)\Gamma(f-a)
\Gamma(1+d-e)\Gamma(1+c-e)
}\right]
}\nonumber\\
-\frac{{}_4F_3\left[
{\displaystyle f-b, f-a, 1+d-e, 1+c-e;
\atop
\displaystyle 1+f-e,f+g-a-b,1+f-a-b}1
\right]}
{\left[
{\displaystyle \sin \pi (1+a+b-f) 
\Gamma(a)\Gamma(b)
\Gamma(g-c)\Gamma(g-d)
\atop
\displaystyle \cdot
\Gamma(1+f-e)
\Gamma(f+g-a-b)\Gamma(1+f-a-b)
}\right]
}.
\end{eqnarray}

We fix $b,c,d,f,g \in
\mathbb{C}$ in such a way that
\begin{equation}
\label{520}
\mbox{Re}(f+g-b-c-d)>0, \quad
\mbox{Re}(f-b)>0.
\end{equation}
Let $a \in \mathbb{C}$ and 
let $e=1+a+b+c+d-f-g$ depend
on $a$. In equation 
(\ref{510}) we let
$|a| \to \infty$. Using Stirling's 
formula (\ref{250}) and the
conditions (\ref{520}), we obtain
\begin{eqnarray}
\label{530}
\frac{{}_3F_2\left[
{\displaystyle b,c,d;
\atop
\displaystyle f,g;}1\right]}
{\Gamma(f)\Gamma(g)
\Gamma(f+g-b-c-d)}
\nonumber\\
=\frac{{}_3F_2\left[
{\displaystyle b,g-c,g-d;
\atop
\displaystyle f+g-c-d,g;}
1\right]}
{\Gamma(f+g-c-d)\Gamma(g)
\Gamma(f-b)}.
\end{eqnarray}
We note that the conditions
(\ref{520}) are needed for
the absolute convergence of
the two ${}_3F_2(1)$ series
in (\ref{530}).
Applying
(\ref{530}) twice yields 
Thomae's identity
\begin{eqnarray}
\label{540}
\frac{{}_3F_2\left[
{\displaystyle b,c,d;
\atop
\displaystyle f,g;}1\right]}
{\Gamma(f)\Gamma(g)
\Gamma(f+g-b-c-d)}
\nonumber\\
=\frac{{}_3F_2\left[
{\displaystyle f-b,g-b,f+g-b-c-d;
\atop
\displaystyle f+g-b-d,f+g-b-c;}
1\right]}
{\Gamma(b)\Gamma(f+g-b-d)
\Gamma(f+g-b-c)}.
\end{eqnarray}
In fact, applying
(\ref{540}) twice gives
(\ref{530}), so that
(\ref{530}) and (\ref{540})
are equivalent.

Next in equation (\ref{510}) we let
$a \to -n$, where $n$ is a
nonnegative integer. Using the fact
that $\lim_{a \to -n}\frac{1}{\Gamma(a)}
=0$ and then formula (\ref{270}) to
simplify the result, we obtain
Bailey's identity
\begin{align}
\label{550}
&{}_4F_3\left[
{\displaystyle -n,b,c,d;
\atop
\displaystyle e,f,g;}
1\right]\nonumber\\
=\frac{(e-b)_n(f-b)_n}
{(e)_n(f)_n}
&{}_4F_3\left[
{\displaystyle -n,b,g-c,g-d;
\atop
\displaystyle 1-n+b-f,1-n+b-e,g;}
1\right], 
\end{align}
which holds provided that
$e+f+g-b-c-d+n=1$.

Thomae's and Bailey's identities
have been shown 
in \cite{FGS}
in a similar way
to be limiting cases of a fundamental
two-term relation 
satisfied by the
function $K(a;b,c,d;e,f,g)$. 

As a final application, in the 
fundamental two-term relation
(\ref{340}) we let $d=g$.
We express the left-hand side
as a Barnes integral 
according to Proposition \ref{310},
and we write
the right-hand side in terms 
of two ${}_4F_3(1)$ series
according to the 
definition
of the $L$ function.
The condition $d=g$ causes
one of the terms on the right-hand 
side to go to zero and the
${}_4F_3(1)$ series in the other
term to be trivially equal to one.
If we simplify the result further using
(\ref{270}),
we obtain
\begin{eqnarray}
\label{560}
\frac{1}{2 \pi i}
\int_t\frac{\Gamma(a+t)\Gamma(b+t)
\Gamma(c+t)\Gamma(1-e-t)\Gamma(-t)}
{\Gamma(f+t)}\,dt\nonumber\\
=\frac{\Gamma(a)\Gamma(b)\Gamma(c)
\Gamma(1+a-e)\Gamma(1+b-e)\Gamma(1+c-e)}
{\Gamma(f-a)\Gamma(f-b)\Gamma(f-c)},
\end{eqnarray}
which holds provided that 
$e+f-a-b-c=1$. The equation (\ref{560})
is precisely 
the statement of Barnes' second lemma.

\end{document}